\documentclass[12pt]{amsart}
\usepackage{amsmath, amscd}
\usepackage{amsfonts}

\begin{document}
\numberwithin{equation}{section}

\def\1#1{\overline{#1}}
\def\2#1{\widetilde{#1}}
\def\3#1{\widehat{#1}}
\def\4#1{\mathbb{#1}}
\def\5#1{\frak{#1}}
\def\6#1{{\mathcal{#1}}}

\newcommand{\Lie}[1]{\ensuremath{\mathfrak{#1}}}
\newcommand{\LieL}{\Lie{l}}
\newcommand{\LieH}{\Lie{h}}
\newcommand{\LieG}{\Lie{g}}
\newcommand{\de}{\partial}
\newcommand{\R}{\mathbb R}
\newcommand{\FH}{{\sf Fix}(H_p)}
\newcommand{\al}{\alpha}
\newcommand{\tr}{\widetilde{\rho}}
\newcommand{\tz}{\widetilde{\zeta}}
\newcommand{\tk}{\widetilde{C}}
\newcommand{\tv}{\widetilde{\varphi}}
\newcommand{\hv}{\hat{\varphi}}
\newcommand{\tu}{\tilde{u}}
\newcommand{\tF}{\tilde{F}}
\newcommand{\debar}{\overline{\de}}
\newcommand{\Z}{\mathbb Z}
\newcommand{\C}{\mathbb C}
\newcommand{\Po}{\mathbb P}
\newcommand{\zbar}{\overline{z}}
\newcommand{\G}{\mathcal{G}}
\newcommand{\So}{\mathcal{S}}
\newcommand{\Ko}{\mathcal{K}}
\newcommand{\U}{\mathcal{U}}
\newcommand{\B}{\mathbb B}
\newcommand{\oB}{\overline{\mathbb B}}
\newcommand{\Cur}{\mathcal D}
\newcommand{\Dis}{\mathcal Dis}
\newcommand{\Levi}{\mathcal L}
\newcommand{\SP}{\mathcal SP}
\newcommand{\Sp}{\mathcal Q}
\newcommand{\A}{\mathcal O^{k+\alpha}(\overline{\mathbb D},\C^n)}
\newcommand{\CA}{\mathcal C^{k+\alpha}(\de{\mathbb D},\C^n)}
\newcommand{\Ma}{\mathcal M}
\newcommand{\Ac}{\mathcal O^{k+\alpha}(\overline{\mathbb D},\C^{n}\times\C^{n-1})}
\newcommand{\Acc}{\mathcal O^{k-1+\alpha}(\overline{\mathbb D},\C)}
\newcommand{\Acr}{\mathcal O^{k+\alpha}(\overline{\mathbb D},\R^{n})}
\newcommand{\Co}{\mathcal C}
\newcommand{\Hol}{{\sf Hol}(\mathbb H, \mathbb C)}
\newcommand{\Aut}{{\sf Aut}(\mathbb D)}
\newcommand{\D}{\mathbb D}
\newcommand{\id}{{\sf id}}
\newcommand{\oD}{\overline{\mathbb D}}
\newcommand{\oX}{\overline{X}}
\newcommand{\loc}{L^1_{\rm{loc}}}
\newcommand{\la}{\langle}
\newcommand{\ra}{\rangle}
\newcommand{\thh}{\tilde{h}}
\newcommand{\N}{\mathbb N}
\newcommand{\kd}{\kappa_D}
\newcommand{\Hr}{\mathbb H}
\newcommand{\ps}{{\sf Psh}}
\newcommand{\Hess}{{\sf Hess}}
\newcommand{\subh}{{\sf subh}}
\newcommand{\harm}{{\sf harm}}
\newcommand{\ph}{{\sf Ph}}
\newcommand{\tl}{\tilde{\lambda}}
\newcommand{\gdot}{\stackrel{\cdot}{g}}
\newcommand{\gddot}{\stackrel{\cdot\cdot}{g}}
\newcommand{\fdot}{\stackrel{\cdot}{f}}
\newcommand{\fddot}{\stackrel{\cdot\cdot}{f}}
\def\v{\varphi}
\def\Re{{\sf Re}\,}
\def\Im{{\sf Im}\,}

\newtheorem{theorem}{Theorem}[section]
\newtheorem{lemma}[theorem]{Lemma}
\newtheorem{proposition}[theorem]{Proposition}
\newtheorem{corollary}[theorem]{Corollary}

\theoremstyle{definition}
\newtheorem{definition}[theorem]{Definition}
\newtheorem{example}[theorem]{Example}

\theoremstyle{remark}
\newtheorem{remark}[theorem]{Remark}
\numberwithin{equation}{section}

\title[Growth Estimates]{growth estimates for pseudo-dissipative holomorphic maps in Banach spaces}
\author[F. Bracci]{Filippo Bracci$^\ast$}
\address{F. Bracci: Dipartimento Di Matematica\\
Universit\`{a} di Roma \textquotedblleft Tor Vergata\textquotedblright\ \\
Via Della Ricerca Scientifica 1, 00133 \\
Roma, Italy} \email{fbracci@mat.uniroma2.it}
\author[M. Elin]{Mark Elin}
\address{M. Elin: Department of Mathematics\\
ORT Braude College \\ 21982 Karmiel, Israel}
\email{mark$\_$elin@braude.ac.il}
\author[D. Shoikhet]{David Shoikhet}
\address{D. Shoikhet: Department of Mathematics\\
ORT Braude College \\ 21982 Karmiel, Israel}
\email{davs@braude.ac.il}
\date\today

\dedicatory{Dedicated to Prof. Simeon Reich for his 65th anniversary}



\thanks{$^{*}$Partially supported by the ERC grant ``HEVO - Holomorphic Evolution Equations'' n. 277691.}

\begin{abstract}
In this paper we introduce a class of pseudo-dissipative
holomorphic maps which contains, in particular, the class of
infinitesimal generators of semigroups of holomorphic maps on the
unit ball of a complex Banach space. We give a growth estimate for
maps of this class. In particular, it follows that
pseudo-dissipative maps on the unit ball of (infinite-dimensional)
Banach spaces are bounded on each domain strictly contained inside
the ball. We also present some applications.
\end{abstract}

\maketitle

\section{Introduction}

Let $(X, \|\cdot\|)$ be a complex Banach space and let $\B:=\{z\in X : \|z\|<1\}$.

\begin{definition}[see \cite{FV, H}]
Let $h: \B \mapsto X$ be a holomorphic map. One says that it has
{\sl unit radius of boundedness} if it is bounded on each subset
strictly inside $\B$.
\end{definition}

The problem to verify whether a holomorphic map has unit radius of
boundedness arises in many aspects of infinite dimensional
holomorphy (see, for example, \cite{FV, H}) as well as in complex
dynamical systems (\cite{ARS, RS-97, RS}). In particular, it plays
a crucial role in the study of nonlinear numerical range and
spectrum of holomorphic maps \cite{H}, in the establishing of
exponential and product formulas for semigroups of holomorphic
maps \cite{RS-SD-96, RS-97}, and the study of flow invariance and
range conditions in the nonlinear analysis \cite{HRS, RS}. It was
specifically mentioned for the class of the so-called
semi-complete vector fields (or, infinitesimal generators) in the
study of their numerical range and Bloch radii \cite{HRS}.

Note that if $h$ is uniformly continuous on $\overline\B$, the
closure of $\B$, then the property of $h$ to be an infinitesimal
generator is equivalent to the nonlinear dissipativeness of $h$
\cite{HRS}.

In this note we consider a wider class of holomorphic maps on $\B$
and establish some growth estimates under weaker restrictions on
its numerical range.

For $\varphi \in X^\ast$ we use the notation $\langle v,
\varphi\rangle := \varphi(v)$. If $v\in X$, we denote by
$v^\ast\in X^\ast$ any element such that
\[
\Re \langle v, v^\ast \rangle =\|v\|^2=\|v^\ast\|^2.
\]
By the Hahn-Banach theorem such an element $v^\ast$ exists, but in general it is not unique. However, if $(X,\langle \cdot, \cdot\rangle)$ is a Hilbert space then $v^\ast$ is unique and it is defined by $v^\ast:=\langle \cdot, v\rangle$.

\begin{definition}\label{def_pd}
We say that a holomorphic map $h: \B \mapsto X$ is
{\sl pseudo-dissipative} if there is $\varepsilon>0$ such that the
convex hull of the set
\[
\Omega_{\varepsilon}(h):=\left\{\langle h(z),z^*\rangle,\
1-\varepsilon<\|z\|<1 \right\}
\]
is not the whole plane $\C$.
\end{definition}

\begin{remark}
If $h$ is uniformly continuous on $\overline\B$, then its
pseudo-dissipativeness actually means that the numerical range of
$h$ (see \cite{H}) lies in a half-plane.
\end{remark}

\begin{remark}
It is clear that the pseudo-dissipativeness of $h$ is equivalent
to the following condition: {\it there is $\varepsilon>0$ such
that
\begin{equation}\label{ineq_pd}
\Re e^{i\theta}\langle h(z), z^*\rangle \le a\|z\|^2 +b\left(1-
\|z\|^2 \right),\quad 1-\varepsilon<\|z\|<1,
\end{equation}
for some real $\theta,\ a$ and $b$.} For some  technical reasons
inequality \eqref{ineq_pd} is more convenient for our further
considerations.
\end{remark}

Let $A: X \mapsto X$ be a continuous linear operator. Then one
defines
\[
m(A):=\inf\{\Re \langle Av, v^\ast\rangle : \|v\|=1\}
\]
and the {\sl numerical radius} of $A$ as
\[
V(A):=\sup \{|\langle Av, v^\ast\rangle| : \|v\|=1\}.
\]

The main result of this paper is the following

\begin{theorem}\label{mainA}
Let $\B$ be the unit ball in a complex Banach space $X$. Let $F$
be a  pseudo-dissipative holomorphic map on $\B$, i.e., for some
real $a,b,\theta$ and $\varepsilon>0$, inequality~\eqref{ineq_pd}
holds for $z,\ 1-\varepsilon<\|z\|<1$. Then

\begin{itemize}
\item[(i)] inequality~\eqref{ineq_pd} holds with $b=\|F(0)\|$ and the
same $a$ and $\theta$ for all $z\in\B$;

\item[(ii)] $m(e^{i\theta}A)\le a$, where $A=DF(0)$;

\item[(iii)]  $F$ has unit radius of boundedness. Moreover, for
all $z\in \B$ the following estimate holds:
\begin{equation}\label{estimate}
\begin{split}
\|F(z) -F(0)\|&\leq   \|z\|\left(|a| +eV(e^{i\theta}A-a\cdot\id)
\right) \\
& + 4\|z\|^2 \|F(0)\|+ 8\|z\|^2
\frac{1-\|z\|\ln2}{(1-\|z\|)^2}\left(a-m(e^{i\theta}A) \right)  \\
& < 4\|F(0)\|\cdot \|z\|^2 + |a|\cdot
\frac{\alpha\|z\|}{(1-\|z\|)^2} + V(A)\cdot
\frac{\beta\|z\|}{(1-\|z\|)^2} \,,
\end{split}
\end{equation}
where $\alpha=\dfrac{8(e\ln2+\ln2+1-e)}{8\ln2-e-1}<3.8$ and
$\beta=\dfrac{8(e\ln2+2-e)}{8\ln2-e}<3.3\,$.
\end{itemize}
\end{theorem}

\begin{corollary}
Let $D$ be a domain in $X,\ \B\subset D$, and let $F: D\mapsto D$
be pseudo-dissipative on $\B$ such that $F(0)=0$ and $DF(0)=A$. If
operator $e^{i\theta}A-a\cdot\id$ is power bounded, then there is
a bounded neighborhood of the origin which is invariant under $F$.
\end{corollary}

Let $G: \B \mapsto X$ be holomorphic. The map $G$ is called an
{\sl infinitesimal generator} if the Cauchy problem
\[
\begin{cases}
\stackrel{\bullet}{z}(t)=G(z(t))\\
z(0)=z_0
\end{cases}
\]
has a solution $[0,+\infty)\ni t\mapsto z(t)$ for all $z_0\in \B$.

We will see below (see Lemma~\ref{L_pd-g}) that a map $G:\B\mapsto
X$ is an infinitesimal generator on $\B$ if and only if it
satisfies the inequality
\[
\Re\langle G(z),z^*\rangle \le b(1-\|z\|^2),\quad
1-\varepsilon<\|z\|<1,
\]
for some $\varepsilon>0$ and $b\in\R$ (cf.,
formula~\eqref{geninf}). Therefore, we obtain from
Theorem~\ref{mainA} that that {\it each holomorphic generator
satisfies the following inequality:
\begin{equation}\label{estimate-g}
\begin{split}
& \|G(z) -G(0)\|\leq 4\|z\|^2 \|G(0)\|  +e \|z\|V(T)
-8m(T)\|z\|^2 \frac{1-\|z\|\ln2}{(1-\|z\|)^2}  \\
& < 4\|z\|^2 \|G(0)\|+
\frac{\beta\|z\|}{(1-\|z\|)^2}V(T),\quad\mbox{where }T=DG(0)\mbox{
and } \beta<3.3\,.
\end{split}
\end{equation}
} In case $X=\C^n$ with some norm $\|\cdot \|$, under the
conditions that $G(0)=0$ and $m(T)<0$, the last estimate has been
proved in \cite[Lemma 2.4]{DGHK2010} (generalizing the previous
result for the case $T=\id$ in \cite[Theorem~1.2]{GHK2002}). In
\cite[Theorem~8]{HRS} a similar growth estimate for holomorphic
maps $G: \B \mapsto X$ with bounded numerical range and such that
$G(0)=0, DG(0)=\id$ is given. Note that estimate~\eqref{estimate}
is more precise than those mentioned above.

%


Also, even in finite dimensional spaces, Theorem \ref{mainA} has some interesting applications. For instance, it can be used to give an answer to the following natural question.

\medskip
{\bf Question:} Let $\{G_n\}$ be a family of infinitesimal generators on the unit ball $\B$ of a complex Banach space. Which are the ``minimal'' possible conditions that guarantee that the family contains a convergent subsequence?
\medskip

In \cite[Lemma 2.2]{BCD2} it is shown that if $\B=\D$ is the unit disc in $\C$, then $\{G_n\}$ contains a convergent subsequence if there exist two points $z\neq w\in \D$ such that $\{G_n\}$ is equibounded at $z$ and $w$. The argument there is strongly based on the so-called Berkson-Porta formula. It is not clear whether and how a similar statement can be prove in higher dimensions. However, Theorem \ref{mainA} allows to prove the following:

\begin{corollary}\label{extract}
Let $D\subset \C^n$ be a bounded balanced convex domain. Let $\{G_n\}$ be a family of infinitesimal generators on $D$. Suppose that there exists $C>0$ such that
\[
\|G_n(0)\|+\|G_n'(0))\|\leq C,
\]
where here we denote by $\|G_n'(0)\|$ the operator norm of the differential of $G_n$ at $0$. Then there exists a subsequence of $\{G_{n}\}$ which converges uniformly on compacta to an infinitesimal generator.
\end{corollary}

\medskip

 This work started when the first and third authors where visiting the
Mittag-Leffler Institute during the program ``Complex Analysis and
Integrable Systems'' in Fall 2011. Both authors thank the
organizers and the Institute for the kind hospitality and the
atmosphere experienced there.

\section{Proofs}\label{two}

If $h$ is a holomorphic map on $\B$, we will write its expansion
at $0$ as $h(z)=h(0)+Tz+\sum_{j\geq 2} Q_j(z)$, where $T: X
\mapsto X$ is the Fr\'echet differential of $h$ at $0$, and $Q_j$
is a continuous homogeneous polynomial of degree $j$ on $X$ (see,
{\sl e.g.}, \cite{Muj}).

We begin by recalling (see \cite[Theorem p.95]{ARS} or \cite{RS})
that $G:\B \mapsto X$ is an infinitesimal generator if and only if
\begin{equation}\label{geninf}
\Re \langle G(z), z^\ast\rangle \leq \Re \langle
G(0),z^\ast\rangle(1-\|z\|^2)\qquad \mbox{for all }z\in \B.
\end{equation}

Note that if the previous formula holds for some $z^\ast$, then it holds for all $z^\ast$.

The following lemma follows immediately from \eqref{geninf}

\begin{lemma}\label{restrict-gen}
Let $G:\B \mapsto X$ be holomorphic. Then $G$ is an infinitesimal
generator if and only if for all $v\in X$ such that $\|v\|=1$ the
holomorphic map
\[
\D\ni \zeta \mapsto \langle G(\zeta v),v^\ast\rangle,
\]
is an infinitesimal generator in the unit disc $\D:=\{\zeta\in \C : |\zeta|<1\}$.
\end{lemma}

In case of strongly convex domains in $\C^n$ the previous lemma holds for the restriction to any complex geodesic (see \cite[Proposition 4.5]{BCD}).

\begin{remark}\label{disco}
By \cite[Corollary 5]{ARS}, a holomorphic map $g: \D \mapsto \C$
is an infinitesimal generator if and only if
\[
g(\zeta)=g(0)-\overline{g(0)}\zeta^2-\zeta q(\zeta),
\]
where $\Re q(\zeta)\geq 0$ for all $\zeta \in \D$.
\end{remark}

To prove our theorem we also need the following lemmata.

\begin{lemma}\label{L_pd-g}
Let $F:\B\mapsto X$ be a pseudo-dissipative holomorphic map. Then
inequality \eqref{ineq_pd} holds for some suitable $\theta$, $a$ and
$b$ for all $z\in\B$. Moreover, the map $G:\B\mapsto X$
defined by
\begin{equation}\label{f_pd-g}
G(z)=e^{i\theta}F(z)-az
\end{equation}
is an infinitesimal generator on $\B$.
\end{lemma}

\begin{proof}
It follows from \eqref{f_pd-g} and \eqref{ineq_pd} that $G$
satisfies the inequality
\begin{equation}\label{ineq_G}
\Re \langle G(z), z^*\rangle \le b\left(1- \|z\|^2 \right)
\end{equation}
for all $z$ with $1-\varepsilon<\|z\|<1$.

Fix any $v\in\partial\B$ and consider the function $g$ defined by
\[
g(\zeta)=\langle G(\zeta v),v^*\rangle,\quad \zeta\in\D.
\]
Actually, according to Lemma~\ref{restrict-gen}, we have to show
that $g$ is an infinitesimal generator on $\D$. To do this, fix
any $w\in\D$, choose $s\in\R$ with $\max(1-\varepsilon,|w|)<s<1$
and define
\[
h_{r,t}(\zeta)=\zeta- t\left(w+rg(\zeta)\right),\quad r>0,\ 0\le
t\le1.
\]
For all $\zeta$ on the circle $|\zeta|=s,$ we have
by~\eqref{ineq_G}
\begin{eqnarray*}
\Re\Bigl(h_{r,t}(\zeta)\bar\zeta\Bigr)=|\zeta|^2
-t\Re(w\bar\zeta)-tr\Re\langle G(\zeta)v,(\zeta v)^*\rangle \ge
s^2-ts|w| -trb(1-s^2).
\end{eqnarray*}
It is clear that for $s$ close enough to $1$, there is $\delta>0$
such that $\Re\Bigl(h_{r,t}(\zeta)\bar\zeta\Bigr)>\delta$ on the
circle $|\zeta|=s$. Therefore, $h_{r,t}$ has no null points on
this circle. Then by the logarithmic residue formula, the number
of null points is a continuous in $t$ function. Since this
function takes natural values, it is constant. Because of
$h_{r,0}(\zeta)=\zeta$, we conclude that the function
$h_{r,1}=\zeta- w-rg(\zeta)$ has a unique null point in~$\D$. Then
Proposition~3.3.1 in \cite{SD} implies that $g$ is a generator.

Now, by~\eqref{geninf}, we have the following inequality:
\[
\Re\langle e^{i\theta}F(z),z^*  \rangle \le a\|z\|^2
+\|F(0)\|(1-\|z\|^2),\quad z\in\B,
\]
which completes the proof.
\end{proof}

\begin{lemma}\label{poly-ort}
Let $G:\B \mapsto X$ be an infinitesimal generator with expansion
$G(z)=G(0)+Tz+\sum_{j\geq 2}Q_j(z)$. Let $v\in X$ be such that
$\|v\|=1$ and let $v^\ast\in X^\ast$ be such that $\langle v,
v^\ast\rangle=\|v\|=\|v^\ast\|=1$. Then
\begin{equation}\label{estim}
\Re \langle Tv, v^\ast\rangle \leq 0.
\end{equation}
Moreover, if $\Re \langle Tv, v^\ast\rangle =0$ then
\begin{equation*}
\begin{split}
\langle Q_2(v), v^\ast\rangle &=-\overline{\langle G(0), v^\ast\rangle} ,\\
\langle Q_j( v), v^\ast\rangle &=0 ,\quad j\geq 3.
\end{split}
\end{equation*}
\end{lemma}

\begin{proof}
let $g: \D \mapsto \C$ be defined as $g(\zeta):=\langle G(\zeta
v), v^\ast\rangle$. By Lemma \ref{restrict-gen} the holomorphic
map $g$ is an infinitesimal generator in $\D$. By Remark
\ref{disco} we can write
$g(\zeta)=g(0)-\overline{g(0)}\zeta^2-\zeta q(\zeta)$, where $\Re
q(\zeta)\geq 0$ for all $\zeta \in \D$. From this we have
\begin{equation}\label{q0}
-q(0)=g'(0)=\langle Tv, v^\ast \rangle,
\end{equation}
and \eqref{estim} follows from $\Re q(0)\geq 0$.

Now, let $q(\zeta)=q(0)+\sum_{j\geq 1} a_j \zeta^j$. Expanding $g$ we see that
\begin{equation}\label{expansion}
\begin{split}
g(0)-q(0)\zeta& -(a_1+\overline{g(0)})\zeta^2 - \sum_{j\geq 3} a_{j-1} \zeta^{j}=g(\zeta)\\&=\langle G(0), v^\ast\rangle + \langle Tv, v^\ast\rangle \zeta +\sum_{j\geq 2} \langle Q_j(v),v^\ast\rangle \zeta^j,
\end{split}
\end{equation}
from which it follows that
\begin{equation}\label{equal}
\begin{split}
\langle Q_2(v), v^\ast\rangle &=-\left(a_1+\overline{\langle G(0), v^\ast\rangle}\right),\\
\langle Q_j(v), v^\ast\rangle &= -a_{j-1}, \quad j\geq 3.
\end{split}
\end{equation}
If $\Re \langle Tv, v^\ast \rangle=0$ then $\Re q(0)=0$, hence $q(\zeta)=i a$ for some $a\in \R$ and $a_j=0$ for all $j\geq 1$, and the statement follows.
\end{proof}

Now we are in good shape to prove our main result:

\begin{proof}[Proof of Theorem \ref{mainA}]
Assertion (i) is already proven in Lemma~\ref{L_pd-g}. Consider
now the map $G:\B\mapsto X$ defined by $G(z)=e^{i\theta}F(z)-az$.
Then $T(=DG(0))=e^{i\theta}A- a\cdot\id$, and assertion~(ii)
follows immediately  by Lemma~\ref{poly-ort}.

In order to prove assertion~(iii), by using Lemma~\ref{L_pd-g}, it
is sufficient to prove inequality~\eqref{estimate-g} for the same
map $G$ being an infinitesimal generator on~$\B$.

For a fixed $v\in X$ with $\|v\|=1$, let $v^\ast\in X^\ast$ be
such that $\langle v, v^\ast\rangle=\|v\|=\|v^\ast\|=1$. It
follows from Lemma~\ref{poly-ort} that $\Re \langle Tv,
v^\ast\rangle \le 0$. Let $g: \D \mapsto \C$ be defined as
$g(\zeta):=\langle G(\zeta v), v^\ast\rangle$. By
Lemma~\ref{restrict-gen}, the holomorphic map $g$ is an
infinitesimal generator in $\D$. According to Remark~\ref{disco},
$g(\zeta)=g(0)-\overline{g(0)}\zeta^2-\zeta q(\zeta)$, where $\Re
q(\zeta)\geq 0$ for all $\zeta \in \D$. The Carath\'eodory
inequalities (\cite{Ca}, see also \cite{Ai}) imply that $|a_j|\leq
2\Re q(0)$ for all $j\geq 1$. Now by \eqref{equal} and \eqref{q0}
we get
\begin{equation}\label{stime2}
\begin{split}
|\langle Q_2(v), v^\ast\rangle| \leq& \|G(0)\| -2\Re \langle Tv, v^\ast\rangle,\\
|\langle Q_j(v), v^\ast\rangle| \leq& -2\Re\langle Tv,
v^\ast\rangle, \quad j\geq 3.
\end{split}
\end{equation}
To proceed, we need L. Harris' inequalities \cite{H}. Namely, if
$P_m: X \mapsto X$ is a continuous homogeneous polynomial of
degree $m\geq 1$ then $\|P_m\|\leq k_m V(P_m)$, where
$k_m=m^{m/(m-1)}$ for $m\geq 2$, $k_1=e$ and
\[
V(P_m)=\sup \{ |\langle P_m(v), v^\ast\rangle |: \|v\|=1\}.
\]
These estimates together with~\eqref{stime2} imply that
\begin{equation}\label{stime}
\begin{split}
\|Tz\|\leq & e \|z\|V(T),\\
\|Q_2(z)\|\leq& k_2\|z\|^2\|G(0)\|-2k_2\|z\|^2 m(T),\\
\|Q_j(z)\|\leq& -2k_j\|z\|^j m(T), \quad j\geq 3.
\end{split}
\end{equation}
Now for all $z\in \B$ we have by \eqref{stime}
\begin{equation}\label{estimate-pol}
\begin{split}
&\|G(z)- G(0)\|\leq  \|Tz\|+\|Q_2(z)\|  +\sum_{j\geq 3}\|Q_j(z)\|\\
&\leq  e \|z\|V(T)  +k_2\|z\|^2\|G(0)\|-2k_2\|z\|^2 m(T)  -2m(T) \sum_{j\geq 3}k_j\|z\|^j \\
&\leq  e \|z\|V(T)  +4\|z\|^2\|G(0)\|  -2m(T) \sum_{j\geq 2} k_j
\|z\|^j.
\end{split}
\end{equation}
To this end, consider the function $k(x)=x^{x/(x-1)},\ x\ge2$.
Since $k''(x)<0$, one concludes that $k(x)\le
k'(2)(x-2)+k(2)=k'(2) x+4-2k'(2)$ with $k'(2)=4(1-\ln 2)$.
Applying this simple fact, we get from~\eqref{estimate-pol}
\begin{equation*}
\begin{split}
\|G(z) - G(0)\|&\leq   e \|z\|V(T)  +4\|z\|^2\|G(0)\| -2m(T) \sum_{j\geq 2}\left( k'(2)  j+4-2 k'(2) \right)\|z\|^j \\
&\le  e \|z\|V(T)  +4\|z\|^2\|G(0)\| -2m(T)\|z\|^2 \left(\frac{ k'(2) (2-\|z\|)}{(1-\|z\|)^2} + \frac{4-2 k'(2) }{1-\|z\|}  \right)\\
&= 4\|z\|^2\|G(0)\|+e\|z\|V(T)-8m(T)\|z\|^2
\frac{1-\ln2\cdot\|z\|}{(1-\|z\|)^2}\,.
\end{split}
\end{equation*}

\end{proof}

Also, we can prove Corollary \ref{extract}:

\begin{proof}[Proof of Corollary \ref{extract}]
By hypothesis, $D$ is the unit ball in $\C^n$ for the Minkowski
norm defined by $D$. Therefore, Theorem \ref{mainA} applies and
Montel's theorem implies that the family $\{G_n\}$ is normal.
Thus, since it cannot be compactly divergent because it is bounded
at the origin, there exists a subsequence $\{G_{n_k}\}$ which
converges uniformly on compacta to a holomorphic map $G: D \mapsto
\C^n$. Applying \eqref{geninf} to each $G_{n_k}$ and passing to
the limit, we get the result.
\end{proof}

\end{document}